\documentclass[a4paper,10pt]{article}
\usepackage{latexsym} 
\usepackage{amsmath,amssymb}
\usepackage{verbatim}
\usepackage[dvips]{graphicx}

\newtheorem{theorem}{Theorem}[section]

\newtheorem{corollary}[theorem]{Corollary}
\newtheorem{proposition}[theorem]{Proposition}

\newtheorem{preexample}[theorem]{Example}

\newtheorem{preremark}{Remark}
\newenvironment{remark}{\begin{preremark}\rm}{\end{preremark}}

\newcommand{\qed }{ \hfill $\Box$ }

\newcommand{\g}{\mathfrak{g}}
\newcommand{\m}{\mathfrak{m}}

\begin{document}
\begin{center}
{\Large Martingales in Reductive Homogeneous spaces\\}

\end{center}

\vspace{0.3cm}

\begin{center}
{\large  Sim\~ao N. Stelmastchuk}  \\

\textit{Departamento de  Matem\'atica, Universidade Estadual do Paran\'a,\\ 84600-000 -  Uni\~ao da Vit\' oria - PR,
Brazil. e-mail: simnaos@gmail.com}
\end{center}

\vspace{0.3cm}

\begin{abstract}
  The subject of this work is to study martingales in a reductive \linebreak homogeneous space with respect to a symmetric connection. Our \linebreak basic idea is to view  homogenous spaces as principal fiber bundles and, thus, to study martingales on homogeneous space with aid of horizontal \linebreak martingales on Lie group. Furthermore, using the stochastic logarithm we give a characterization of martingales on homogenous space. To end, we study the martingales in spheres $S^{n}$ and $SL(n,\mathbb{R})/SO(n,\mathbb{R})$, $n \geq 2$.
\end{abstract}


\vspace{0.3cm} \noindent {\bf MSC2010 subject classification:} 22F30, 58J65, 60H30, 60G48

\section{Introduction}

The study of martingales in Lie groups and homogeneous space is very rich. The reason for this is the deep connection between the differential geometry and  algebra of Lie groups and homogenous spaces. It allows us to describe some stochastic process in Lie group or homogeneous space from its Lie algebra. We refer the reader to Hakim-Dowek-L\'epingle \cite{hakim} and Liao \cite{liao} to understand as this is done.


In \cite{arnaudon2}, Arnaudon characterized the martingales in the reductive \linebreak homogenous space $G/H$ with respect to the Canonical connection. However, for an arbitrary $G$-invariant connection on $G/H$ the characterization of \linebreak martingales are not known. Therefore, in this work, the aim is to study the martingales with respect to an arbitrary symmetric $G$-invariant connection in $G/H$.

The basic idea is to consider the principal fiber bundle $G(G/H,H)$ such as example 11.1 described by Kobayshi-Nomizu in \cite{kobay}. Thus, the  It\^o's \linebreak stochastic exponential,  which the author defined in \cite{stelmastchuk}, is used to \linebreak characterize the horizontal martingales in $G$. Consequently, the martingales in $G/H$ are characterized. Furthermore, a useful way to study the \linebreak martingales in $G/H$ is obtained from the stochastic logarithm( see Theorem \ref{logarithmcondition}). To end, the latter allows us to study the martingales in symmetric spaces: $S^{n}$  and $SL(n,\mathbb{R})/SO(n,\mathbb{R})$, $n\geq 2$.

We organize this paper in the following: in section 2 we have compiled some facts on Stochastic Analysis on manifolds; in section 3 we introduce the geometric structure on $G(G/H,H)$, more specifically, we choose the kind of symmetric $G$-invariant connections on $G/H$ and left invariant connections on $G$ that we are going to use in the paper; In section 4 our main results are stated and proved; In section 5 applications are indicated.

\section{Stochastic calculus}

In this work we use freely the concepts and notations of Emery \cite{emery1} and Kobayashi and Nomizu \cite{kobay}. From now on the adjective smooth means $C^{\infty}$.

Let $(\Omega, \mathcal{F},(\mathcal{F}_{t})_{t\geq0}, \mathbb{P})$ be a probability space which satisfies the usual \linebreak hypotheses (see for example \cite{emery1}). Our basic assumption is that every stochastic process is continuous.

Let $M$ be a smooth manifold and $X_{t}$ a continuous stochastic process with values in $M$. We call $X_{t}$ a semimartingale if, for all $f$ smooth function, $f(X_{t})$ is a real semimartingale.

Let $M$ be a smooth manifold with a symmetric connection $\nabla^{M}$, $X_t$ a \linebreak semimartingale with values in $M$, $\theta$ a 1-form on $M$ and $b$ a bilinear form on $M$. We denote by $\int_0^t \theta d^{M} X_s$ the It\^o's integral of $\theta$ along $X_t$, by $\int_0^t \theta \delta X_s$ the Stratonovich's integral of $\theta$ along $X_t$ and by $\int_0^t b(dX_s,dX_s)$ the \linebreak quadratic integral of $b$ along $X_t$. We recall that $X_t$ is a $\nabla^{M}$-martingale if $\int_0^t \theta d^ {M} X_s$ is a real local martingale for any 1-form $\theta$ on $M$.

It is possible to show a formula to convert the Stratonovich's integral to the It\^o's integral:
\begin{equation}\label{conversion}
  \int_0^t\theta \delta X_s=\int_0^t\theta d^{M} X_s+\frac{1}{2}\int_0^t\nabla^M\theta\;(dX_s,dX_s).
\end{equation}

Let $M$ and $N$ be smooth manifolds endowed with symmetric connections $\nabla^{M}$ and $\nabla^{N}$, respectively, and $F:M \rightarrow N$ a smooth map. Catuogno in \cite{catuogno} shows the following version for the It\^o's formula on smooth manifolds, which will be said the geometric It\^o's formula:
\begin{equation}\label{ito-ito}
  \int_0^t\theta \;d^{N} F(X_s)=\int_0^tF^*\theta \;d^{M} X_s +\frac{1}{2}\int_0^t\beta_F^*\theta \;(dX_s,dX_s),
\end{equation}
where $\beta_F$ is the second fundamental form of $F$ and $\theta$ is a 1-form on $N$.

\section{Connections on homogeneous spaces}

In this section, we introduce the notations and results on Lie groups and homogeneous spaces that are necessary for us. We begin by introducing the kind of homogeneous spaces. The classical work here is \cite{helgason}.

Let $H$ be a closed Lie subgroup of $G$. Let $\mathfrak{g}$ and $\mathfrak{h}$ denote the Lie algebras of $G$ and $H$, respectively. We assume that the homogeneous space $G/H$ is reductive, that is, there is a subspace $\mathfrak{m}$ of $\mathfrak{g}$ such that $\mathfrak{g} = \mathfrak{h} \oplus \mathfrak{m}$ and $Ad(H)(\m)\subset \m$. Let us denote by $\pi$ the natural mapping of $G$ onto the space $G/H$ of the cosets $gH$, $g \in G$.
There is a neighborhood $U$ at $0$ such that  there is a diffeomorphism from $U$ into $N=\pi\exp(U)$.

For each $a \in G$ we denote the left translation $\tau_{a}: G/H \rightarrow G/H$ which is given by $\tau_{a}(gH) = agH$.  It is direct that if $L_{a}$ are the left translation on $G$, then $\pi \circ L_{a} = \tau_{a} \circ \pi$. Furthermore, since $L_{a}$ is a diffeomorphism, we have
\begin{equation}\label{decompositionTG}
  TG = TG_{\mathfrak{h}} \oplus TG_{\mathfrak{m}},
\end{equation}
where $TG_{\mathfrak{h}}:= \{ (L_{a})_{*e}\mathfrak{h}; \forall\, a \in G \}$ and  $TG_{\mathfrak{m}}:= \{ (L_{a})_{*e}\mathfrak{m};\forall\, a \in G \}$. We denote the horizontal projection of $TG$ onto $TG_{\mathfrak{m}}$ by $\mathbf{h}$.

Being $(G,G/H,H)$ a $H$-principal fiber bundle, Theorem 11.1 in \cite{kobay} assures that vertical part of the Maurer-Cartan on $G$, which is denoted by $\omega$, is a connection form associated to connection $TG_{\m}$ in (\ref{decompositionTG}). Thus the horizontal lift from $G/H$ into $G$ is denoted by $(\cdot)^h$.

Set $A \in \m$. The left invariant vector field $\tilde{A}$ on $G$ is given by $\tilde{A}(g) = L_{g*}A$ and the $G$-invariant vector field $A_{*}$ on $G/H$ is defined by $A_{*} =\tau_{g*}A$. It is clear that $\tilde{A}$ is the horizontal lift vector field of $A_{*}$.


If $\nabla^{G/H}$ is a $G$-invariant connection on $G/H$, then Theorem 8.1 in \cite{nomizu} assures the existence of a unique $Ad(H)$-invariant bilinear map $\beta: \m \times \m \rightarrow \m$ such that
\[
  (\nabla^{G/H}_{A_*}B_*)_{o} = \beta(A,B), \ \ A, B \in \m.
\]
We can construct a complete left invariant connection $\nabla^{G}$ on $G$ by taking its associated bilinear map $\alpha: \g \times \g \rightarrow \g$ satisfying \begin{equation}\label{horizontalcondition}
  \mathbf{h}\alpha(A,B) = \beta(A,B), \ \ \mathrm{for} \ \ A,B \in \m.
\end{equation}

\begin{proposition}\label{propconn1}
  Let $\nabla^{G/H}$ and $\nabla^{G}$ be connections on $G/H$ and $G$, respectively, satisfying the condition (\ref{horizontalcondition}). Then
  \begin{enumerate}
    \item For vector fields $X,Y$ on $G/H$ we have $\mathbf{h}(\nabla^{G}_{X^h}Y^h) = (\nabla^{G/H}_{X}Y)^h$.
    \item Denoting by $\beta_{\pi}$ the second fundamental form of $\pi$ it follows that \linebreak $\beta_{\pi}(U,V) = 0$, where $U,V$ are horizontal vector fields on $G$.
  \end{enumerate}
\end{proposition}
\begin{proof}
  {\it 1.} Set $A, B \in \m$. Then there exist left invariant vectors fields $\tilde{A}, \tilde{B}$ on $\exp(U)$ and $G$-invariant vector fields $A_{*}, B_{*}$ on $N$. It is clear that $\tilde{A}, \tilde{B}$ are horizontal and $\pi_{*}(\tilde{A})= A_{*}$ and $\pi_{*}(\tilde{B}) = B_{*}$. By construction of $\nabla^{G}$, for $g \in \exp(U)$,
  \[
    \pi_{*}(\nabla^{G}_{\tilde{A}}\tilde{B})(g) = \pi_{*}L_{g*} \alpha(A,B) = \tau_{g*}\pi_{*}\alpha(A,B) = \tau_{g*}\beta(A,B) = (\nabla^{G/H}_{A_{*}}B_{*})(\pi(g)).
  \]
  This gives $\mathbf{h}(\nabla^{G}_{\tilde{A}}\tilde{B}) = (\nabla^{G/H}_{A_{*}}B_{*})^h$. A simple argument using properties of connections shows that the  result holds for any vector fields  $X, Y$ on $N$. The same conclusion can be drawn for any vector fields $X,Y$ on $G/K$, because it is sufficient to translate they to $N$ and to use the $G$-invariance of $\nabla^{G/H}$.\\
  {\it 2.} Given $U$ and $V$ horizontal vector fields on $G$, by definition of the second fundamental form, we have
  \[
    \beta_{\pi}(U,V) = \nabla^{G/H}_{\pi_{*}U}\pi_{*}V - \pi_{*}\nabla^{G}_{U}{V}=0,
  \]
  which follows from item {\it 1.}
  \qed
\end{proof}

From now on we make the following assumptions: the left invariants connections $\nabla^{G}$ on $G$ will be complete.

\section{Martingales in homogeneous space}

In this section, we want to characterize the martingales in a reductive homogenous space $G/H$. To do this, we first endow the Lie Groups $G$ with a left invariant connection $\nabla^{G}$ and $\mathfrak{g}$ with the flat connection $\nabla^{\mathfrak{g}}$.

In \cite{stelmastchuk}, the author defines the It\^o's stochastic exponential with respect to $\nabla^G$ and $\nabla^{\g}$ as the solution of the It\^o's  stochastic differential equation
\begin{equation}\label{exponential}
  d^{\nabla^{G}}X_{t} = L_{(X_{t})*}(e) dM_t,\ \ X_{0} = e,
\end{equation}
where $M_t$ is a semimartingale in $\g$. For simplicity, we call $e^{G}(M_t)$ of It\^o's exponential. In \cite{stelmastchuk}, we have the following results about the It\^o's exponential.

\begin{theorem}\label{teomart1}
  Given a semimartingale $X_t$ in $G$, there exists a unique semimartingale $M_t$ in $\g$ such that $X_t = e^{G}(M_t)$.
\end{theorem}

\begin{theorem}\label{martingalesinG}
  Let $\nabla^{G}$ be a left invariant connection on $G$ and $\nabla^{\g}$ the flat connection on $\g$. The $\nabla^{G}$-martingales in $G$ are exactly the process $e^{G}(M_t)$ where $M_t$ is a local martingale in $\mathfrak{g}$.
\end{theorem}

Our first purpose is to apply Theorems \ref{teomart1} and \ref{martingalesinG} in the study of horizontal martingales in $G$. In fact, our idea is to view $\pi:G \rightarrow G/H$ as a $H$-principal fiber bundle and to use the horizontal lift of semimartingales due to Shigegawa in \cite{shig}. In a nutshell, if $X_{t}$ is a semimartingale in $G/H$, there is a unique horizontal lifting $Y_{t}$ in $G$ such that $\pi(Y_{t}) = X_{t}$ and $\int_0^t \omega \delta Y_{s} = 0$ (see Theorem 2.1 in \cite{shig}), where $\omega$ is the vertical part of Maurer-Cartan form on $G$ associated with the horizontal distribution $TG_{\m}$ (\ref{decompositionTG}). A $\nabla^{G}$-martingale $Y_{t}$ is called a horizontal martingale if it satisfies $\int_0^t \omega \delta Y_{s} = 0$.

The next Proposition allows us to consider only  $\nabla^{G/H}$-martingales with initial condition $o$, that is, $X_{0}=o$, where $o = H$ is the origin in $G/H$.

\begin{proposition}
  Let $\nabla^{G/H}$ and $\nabla^{G}$ be connections on $G/H$ and $G$, respectively, satisfying the condition (\ref{horizontalcondition}). If $X_{t}$ is a $\nabla^{G/H}$-martingale in $G/H$ such that $X_{0} = \pi(Y_{0})$, where $Y_{0}$ is random variable in $G$, then so is $Z_{t}= \tau_{Y_{0}^{-1}}X_{t}$.
\end{proposition}
\begin{proof}
  Let $X_{t}$ be a $\nabla^{G/H}$-martingale and $Y_{t}$ its horizontal lift in $G$. Taking a 1-form $\theta$ on $G/H$ we get
  \[
  \int_0^t \theta d^{G/H}Z_{s}  =   \int_0^t \theta d^{G/H}\tau_{Y_{0}^{-1}}X_{s} = \int_0^t \theta d^{G/H}\tau_{Y_{0}^{-1}}\pi(Y_{s}) = \int_0^t \theta d^{G/H}\pi(L_{Y_{0}^{-1}}Y_{s}).
  \]
  From the geometric It\^o's formula (\ref{ito-ito}) and Proposition \ref{propconn1} we see that
  \begin{eqnarray*}
    \int_0^t \theta d^{G/H}Z_{s}
    & = & \int_0^t \pi^{*}\theta d^{G}(L_{Y_{0}^{-1}}Y_{s}) + \frac{1}{2}\int_0^t \beta_{\pi}^*\theta (d(L_{Y_{0}^{-1}}Y_{s}),d(L_{Y_{0}^{-1}}Y_{s}))\\
    & = & \int_0^t \pi^{*}\theta d^{G}(L_{Y_{0}^{-1}}Y_{s}).
  \end{eqnarray*}
  Proposition 3.2 in \cite{stelmastchuk} now assures that
  \[
    \int_0^t \theta d^{G/H}Z_{s} =  \int_0^t \theta \pi_{*} L_{Y_{0}^{-1}*}d^{G}Y_{s} = \int_0^t \theta \tau_{Y_{0}^{-1}*}\pi_{*} d^{G}Y_{s}.
  \]
  Again, from the geometric It\^o's formula (\ref{ito-ito}) and Proposition \ref{propconn1} we conclude that
  \[
    \int_0^t \theta d^{G/H}Z_{s}  = \int_0^t \tau_{Y_{0}^{-1}}^{*}\theta d^{G/H}\pi(Y_{s})
    = \int_0^t \tau_{Y_{0}^{-1}}^{*}\theta d^{G/H}X_{s}.
  \]
  Since $X_{t}$ is a $\nabla^{G/H}$-martingale, it follows that $Z_{t}$ is a $\nabla^{G/H}$-martingale. \qed
\end{proof}

We can now rephrase Theorem \ref{martingalesinG} as follows.

\begin{proposition}\label{horizontalmartingale}
  Let $\nabla^{G/H}$ and $\nabla^{G}$ be connections on $G/H$ and $G$, respectively, satisfying the condition (\ref{horizontalcondition}). If $Y_{t}$ is a horizontal martingale in $G$, then $Y_t = \exp(M_t)$ for some semimartingale $M_t \in \m$.
\end{proposition}
\begin{proof}
  Let $Y_{t}$ be a horizontal martingale in $G$. By Theorem \ref{teomart1}, there is a unique local martingale $N_t$ in $\mathfrak{g}$ such that $d^{G} Y_{t} = L_{Y_{t}*} dN_t$. Let $\{H_{1},\ldots, H_{n}\}$ be a basis on $\mathfrak{g}$ such that $\{H_{\kappa}, \kappa=1, \ldots, r\}$ is a basis for $\mathfrak{m}$.  If we write $N_t= \sum_{\kappa=1}^{r}  N^{\kappa}_tH_{\kappa} + \sum_{j=r+1}^{n} N^{j}_tH_{j}$, then $d^{G} Y_{t} = dN^{\kappa}_t U^{\kappa}_{t} + dN^{j}_tU^{j}_{t}$, where $U^{i}_{t} = L_{Y_{t}*}H_{i}, i=1,\ldots,n$. From formula to convert Stratonovich's integral to the It\^o's integral (\ref{conversion}) and the fact that $Y_t$ is a horizontal martingale we obtain
  \begin{eqnarray*}
    0 = \int_0^t \omega(\delta Y_{s})
    & = & \int_0^t \omega(d^{G} Y_{s}) +\frac{1}{2}\int_0^t \nabla^{G}\omega (dY_s,dY_s) \\
    & = & N^{j}_tH^{j} + \frac{1}{2}\int_0^t \mathbf{v}\nabla^{G}(dN_s,dN_s) \\
    & = & N^{j}_tH^{j} + \frac{1}{2}\int_0^t \mathbf{v}\alpha(dN_s,dN_s) \\
    & = & N^{j}_tH^{j} + \frac{1}{2}\int_0^t \alpha^{j}(dN_s,dN_s)H^{j},
  \end{eqnarray*}
  where $\alpha$ is the bilinear form associated to $\nabla^{G}$. Since $H_{r+1},\ldots, H_{n}$ are linearly independent, it follows that $N^{j}_t= - \frac{1}{2}\int_0^t \alpha^{j}(dN_s,dN_s)$. However, each $N^{j}_t$, $  j=r+1, \ldots, n$, is a real local martingale. We thus conclude that they are null. Therefore, it is sufficient to write $M_t = \sum_{\kappa=1}^{r}  N^{\kappa}_tH_{\kappa}$, and the proof is complete.
\end{proof}

As a direct consequence of the proof of the Proposition above we have a characterization for horizontal martingales.

\begin{corollary}\label{horinzontalito}
  Under hypothesis of Proposition \ref{horizontalmartingale}, a $\nabla^{G}$-martingale $Y$ in $G$ is horizontal if and only if $\int_0^t \omega d^{G}Y_s = 0$.
\end{corollary}

We now relate the martingales in $G/H$ with the horizontal martingales in $G$.

\begin{proposition}\label{onetoonehorizontalmartingale}
  Let $\nabla^{G/H}$ and $\nabla^{G}$ be connections on $G/H$ and $G$, respectively, satisfying the condition (\ref{horizontalcondition}). A semimartingale $X_{t}$ in $G/H$ is a $\nabla^{G/H}$-martingale if, and only if, its horizontal lift $Y_t$  is a horizontal martingale in $G$.
\end{proposition}
\begin{proof}
  Let $X_{t}$ be a semimartingale in $G/H$ and $Y_{t}$ its horizontal lift in $G$. Consider a 1-form $\theta$ in $T^*(G/K)$. Suppose that $Y_t$ is a horizontal martingale. From Proposition \ref{propconn1} and the geometric It\^o's formula (\ref{ito-ito}) we obtain
  \begin{equation}\label{onetoonehorizontalmartingale1}
    \int_0^t \theta d^{G/H}X_{s} = \int_0^t \theta d^{G/H} \pi(Y_{s}) = \int_0^t ({\pi}^{*}\theta) d^{G}Y_{s}.
  \end{equation}
  Consequently, $X_t$ is a $\nabla^{G/H}$-martingale.

  Conversely, suppose that $X_t$ is a $\nabla^{G/H}$-martingale in $G/H$. Take $\eta$ in $\Gamma(T^*G)$. Since the $\mathcal{C}^{\infty}$-module $\Gamma(T^*G)$ is generated by $\omega$ and by the \linebreak differential forms $\pi^* \alpha$ with $\alpha \in \Gamma(T^*G/H)$, we have that $\eta$ is a linear \linebreak combination of  differential forms $f\pi^* \alpha$ and $h\omega$ with $f,h \in \mathcal{C}^{\infty}(G)$.  From \linebreak Corollary \ref{horinzontalito} we deduce that $\int_0^t h\omega d^{G}Y_s =\int_0^t h(Y_s) d(\int_0^s \omega d^{G}Y_r) = 0$. Then
  \[
    \int_0^t \eta d^{G}Y_s = \int_0^t f \pi^* \alpha d^{G}Y_s=\int_0^t (f \alpha) d^{G/H}X_s,
  \]
  where we used (\ref{onetoonehorizontalmartingale1}) in the last equality. Thus it shows that $Y_t$ is a horizontal martingale.\qed
\end{proof}

Since we know horizontal martingales in $G$, Proposition \ref{horizontalmartingale}, and we know the one-to-one correspondence between martingales in $G/H$ and horizontal \linebreak martingales in $G$, Proposition \ref{onetoonehorizontalmartingale}, we are in a position to describe martingales in $G/H$.

\begin{theorem}\label{martingalesinGK}
  Let $\nabla^{G/H}$ and $\nabla^{G}$ be connections on $G/H$ and $G$, respectively, satisfying the condition (\ref{horizontalcondition}). If a semimartingale $X_{t}$ is a $\nabla^{G/H}$-martingale in $G/H$, then it is written as $\pi(e^{G}(M_t))$, where $M_t$ is a local martingale in $\mathfrak{m}$.
\end{theorem}
\begin{proof}
  Let $X_{t}$ be a $\nabla^{G/H}$-martingale in $G/H$. Then Proposition \ref{onetoonehorizontalmartingale} \linebreak assures that its horizontal lift $Y_{t}$ is a horizontal martingale. Now $Y_{t}$ is  written  as \linebreak $Y_{t} = e^{G}(M_t)$ for a local martingale $M_t$ in $\m$, by Proposition \ref{horizontalmartingale}. We thus get $X_{t} = \pi (e^{G}(M_t))$.\qed
\end{proof}

Our next purpose is given other characterization of $\nabla^{G/H}$-martingales in $G/H$. Before we introduce the stochastic logarithm due to Hakim-Dowek-L\'epingle \cite{hakim}. In fact, if $Y_t$ is a semimartingale in $G$, then the stochastic logarithm of $Y_t$, denoted by $L(Y_{t})$, is the unique solution of the Stratonovich's stochastic differential equation
\[
  \delta L(Y_{t})= L_{(Y_{t})^{-1}*}(Y_{t}) \delta Y_{t} \ \ \textrm{and} \ \ Y_{0}=e.
\]
It is simple to see that $L(Y_{t}) = \int_{0}^{t}\omega_G \delta Y_{s}$, where $\omega_G$ is the Maurer-Cartan form on $G$.

\begin{theorem}\label{logarithmcondition}
  Let $\nabla^{G/H}$ and $\nabla^{G}$ be connections on $G/H$ and $G$, respectively, satisfying the condition (\ref{horizontalcondition}). Let $X_{t}$ be a semimartingale in $G/H$ and $Y_t$ its horizontal lift in $G$. Then $X_{t}$ is a $\nabla^{G/H}$-martingale if and only if
  \[
    L(Y_{t})+ \frac{1}{2}\int_{0}^{t}\beta(L(Y_{s}),L(Y_{s}))
  \]
  is a local martingale in $\m$.
\end{theorem}
\begin{proof}
  Let $X_{t}$ be a semimartingale in $G/H$ and $Y_t$ its horizontal lift in $G$. Proposition \ref{horizontalmartingale} assures that $X_t$ is a $\nabla^{G}$-martingale if and only if $Y_t$ is a horizontal martingale. But Proposition 5.1 in \cite{stelmastchuk} says that $Y_t$ is a horizontal martingale if and only if
  \[
    L(Y_{t})+ \frac{1}{2}\int_{0}^{t}\alpha(L(Y_{s}),L(Y_{s})).
  \]
  is a local martingale in $\g$. Taking in account that $L(Y_{s})$ belongs to $\m$ and the choice of the bilinear form $\alpha$ we obtain the proof.\qed
\end{proof}

\begin{corollary}
  Under the hypothesis of Theorem \ref{logarithmcondition}, if moreover $\beta(A,A) = 0$ for all $A \in \m$, then $X_t$ is a $\nabla^{G/H}$-martingale if and only if $L(Y_t)$ is a local martingale.
\end{corollary}

\section{Applications}

In this section, we wish to apply Proposition \ref{logarithmcondition} to characterize the martingales in $S^{n}$ and $Sl(n \mathbb{R}))/SO(n\mathbb{R})$, $n \geq 2$

\subsection{Martingales in the spheres $S^{n}, n \geq 2$}


We begin by recalling that the sphere $S^{n}, n\geq 2$, can be viewed as the \linebreak symmetric space $SO(n+1)/SO(n)$ with the usual metric induced of $\mathbb{R}^{n+1}$ (see for example Example 3.65(a) in \cite{warner}). In addition, $SO(n+1)/SO(n)$ is a \linebreak reductive  homogeneous space with the reductive decomposition given by \linebreak $\mathfrak{so}(n+1) = \mathfrak{so}(n) + \m$, where $\m$ is the subspace of all $n \times n$ matrices of the form
\begin{equation}\label{matrixm}
  \left(
  \begin{array}{cc}
    0 & -x^{t}\\
    x & 0_{n}
  \end{array}
  \right),
\end{equation}
where $x=(x_{1}, \ldots, x_{n})$ is a column vector in $\mathbb{R}^{n}$ and $0_{n}$ is the $n\times n$ zero matrix.


Let $X_{t} \in S^{n}$ be a semimartingale and $Y_t$ its horizontal lift in $SO(n+1)$. By Theorem \ref{logarithmcondition}, $X_t$ is a martingale in $S^{n}$, with respect to usual metric induced of $\mathbb{R}^{n+1}$, if the stochastic logarithm $L(Y_t)$ is a local martingale in $\m$. Here, \linebreak $L(Y_{t}) = \int_0^t \omega_{SO(N+1)}\delta Y_{s}$, where $\omega_{SO(N+1)}$ is the Maurer-Cartan form on \linebreak $SO(N+1)$. For simplicity, let us write $\omega_{SO}$ instead of $\omega_{SO(n+1)}$.


In order to know when $L(Y_{t})$ is a local martingale in $\m$ we adopt the matrix coordinate system. In fact, since $SO(n+1) \subset GL(n+1)$, we use the global coordinates of $GL(n+1)$ to represent the matrix in $SO(n+1)$, namely, if $A \in SO(n+1)$, then $A=(a_{ij})$, $i,j=1, \ldots n+1$.


Let us denote $E_{ij}$ the $n\times n$ matrix with value 1 in the $(i,j)$ entry, $i\neq j$. Then $\{E_{j1} - E_{1j},j=2, \ldots n+1\}$ is a basis of $\m$. Recalling (\ref{decompositionTG}) we write
\[
  TSO(n+1)= TSO(n+1)_{\mathfrak{so}(n)} \oplus TSO(n+1)_{\m},
\]
where $TSO(n+1)_{\m}(A)= \{AV: V \in \m\}$ for $A$ in $SO(n+1)$. Then using the matrix coordinate system we have that $\{\partial y_{j1}(A) - \partial y_{1j}(A), j=2, \ldots n+1\}$ is a basis for $TSO(n+1)_{\m}(A)$ and, consequently, $\{\frac{1}{2}(dy_{j1} - dy_{1j}), j=2, \ldots n+1\}$ is its the dual basis.


Taking a horizontal vector field $V$ on $SO(n+1)$, for $A \in SO(n+1)$, we have $V(A) \in TSO(n+1)_{\m}(A)$ and
\[
  V(A)= a_{12}(V(A))(\partial y_{21} - \partial y_{12}) + \ldots + a_{1(n+1)}(V(A))(\partial y_{(n+1)1} - \partial y_{1(n+1)}).
\]
where  $a_{1l}(V(A)) = \frac{1}{2}(dy_{j1} - dy_{1j})(V(A))$, for $j=2, \ldots, n+1$. Since $\omega_{SO}(A)(V)=AV$, it follows that
\[
  \omega_{SO}(V(A)) = a_{12}(V(A))(E_{21} - E_{12}) + \ldots + a_{1(n+1)}(V(A))(E_{(n+1)1} - E_{1(n+1)}).
\]
Write, in coordinates, $X_{t}=(X^{1}_t, \ldots, X^{n+1}_t)$ and $Y_t = (Y^{ij}_t)$. Thus, as $Y_t$ is a horizontal martingale we have
\[
  \int_0^t \omega_{SO} \delta Y_{s} =\sum_{j=2}^{n+1}\frac{1}{2}(X^{j}_t-Y^{1j}_t) (E_{j1} - E_{1j}),
\]
where we used the fact that $Y^{j1}_t= X^j_t$, $j=2, \ldots, n+1$.

\begin{proposition}
  Let $X_{t}=(X^{1}_t, \ldots, X^{n+1}_t)$ be a semimartingale in $S^{n}, n \geq 2$ and $Y_t = (Y^{ij}_t)$ its horizontal lift in $SO(n+1)$. Then $X_t$ is a martingale in $S^n$ with respect to the usual metric  induced of $\mathbb{R}^{n+1}$, if and only if,
  \[
    (X^2_t-Y^{12}_t, \ldots, X^{n+1}_t-Y^{1n+1}_t)
  \]
  are real local martingales.
\end{proposition}

\begin{remark}
 An analogous construction can be used to characterize the martingales in the Grassman manifold with metric given by Killing form.
\end{remark}

\subsection{Martingales in $Sl(n,{\Bbb R})/SO(n,{\Bbb R}), \, n\geq 2$}


Let $Sl(n,\mathbb{R})$ be the special linear group. Take ${\frak sl}(n, {\Bbb R})$ the Lie algebra of all $n \times n$ matrices over ${\Bbb R}$ with vanishing trace, this is the Lie algebra of ${\rm Sl}(n,\mathbb{R})$.  Choose its Cartan decomposition given by
\[
  {\frak sl}(n, {\Bbb R})={\frak so}(n,{\Bbb R}) \oplus {\frak s},
\]
where
\[
  {\frak so}(n,{\Bbb R})= \{ X \in {\frak sl}(n, {\Bbb R}): X^{t}=-X \} \mbox{ and } {\frak s}=\{X \in {\frak sl}(n, {\Bbb R}): X^{t}=X \}.
\]

Now take ${\frak a}$ a maximal abelian subalgebra of ${\frak s}$ given by all diagonal matrices  with vanishing trace. Also with canonical choices we    have the following Iwasawa decomposition
\[
  {\frak sl}(n, {\Bbb R})= {\frak so}(n,{\Bbb R}) \oplus {\frak a} \oplus {\frak n}^{+} ,
\]
where ${\frak n}^{+}$ is the subalgebra of all upper triangular matrices with null entries on the diagonal.

Hence, the Iwasawa decomposition of the Lie group ${\rm Sl}(n,\mathbb{R})$ is given by ${\rm Sl}(n,\mathbb{R})= {\rm SO}(n,\mathbb{R})  \cdot A \cdot N^{+}$, where
\[
  A = \left\{ \textrm{diag}(a_1, \ldots, a_{n-1},a_n) : a_1\cdot a_2\cdot \ldots \cdot a_{n-1}\cdot a_n = 1 \right\}
\]
and
\[
  N^{+} = \left\{(y_{ij}): y_{ii}=1, y_{ij} = 0,\, i>j,  y_{ij} \in {\Bbb R}, 1<i<j\leq n \right\} .
\]

It is well-known that $A = \mathrm{exp} {\frak a}$, $N^+ = \mathrm{exp}{\frak n}^{+}$ and that $Sl(n,{\Bbb R})/SO(n,{\Bbb R})$ is a symmetric space. Thus the Levi-Civita connection for any invariant metric on $Sl(n,{\Bbb R})/SO(n,{\Bbb R})$ is the Canonical connection (see for example Theorem 11.3 \cite{kobay}).

Let $X_t$ be a semimartingale in $Sl(n,{\Bbb R})/SO(n,{\Bbb R})$ and $Y_t$ its horizontal lift in $Sl(n,{\Bbb R})$. Theorem \ref{logarithmcondition} assures that $X_t$ is a martingale if and only if $L(Y_t)$ is a local martingale in $\m = \mathfrak{a} \oplus \mathfrak{n}^+$. Since $Y_t$ is a horizontal semimartingale in $Sl(n,{\Bbb R})$, we can write
\[
  Y_t =  A_t \cdot N_t,
\]
where $A_t \in A$ and $N_t \in N^+$. Differentiate the equality above with respect to the Stratonovich's differential we have to
\[
  \delta Y_t = L_{N_t*}\delta A_t \oplus R_{A_t*}\delta N_t.
\]
Thus we have that
\begin{equation}\label{eqslnson1}
  L(Y_t) =    L(A_t) \oplus \int_{0}^t Ad(A_s^{-1})\delta L(N_s),
\end{equation}
where we use the fact that $\omega R_{a*}(X) = Ad(a^{-1})\omega(X)$.

\begin{proposition}
  Let $X_t$ be a semimartingale in $Sl(n,{\Bbb R})/SO(n,{\Bbb R})$. Then $X_t$ is a martingale with respect to Canonical connection if, and only if, so is $A_t$ in $A$ and
  \[
   \int_{0}^t Ad(A_s^{-1})\delta L(N_s)
  \]
  is local martingale in $\mathfrak{n}^+$.
\end{proposition}
\begin{proof}
  It is a direct application of Theorem \ref{logarithmcondition} and equation (\ref{eqslnson1}). It is sufficient observe that $A_t$ is a semimartingale in $A \cdot N^+$. Since $\beta =0$, Proposition 5.1 in \cite{stelmastchuk} assures that $A_t$ is a martingale with respect to Canonical connection in $A$ if and only if $L(A_t)$ is a local martingale in $\mathfrak{a} \oplus \mathfrak{n}^+$. \qed
\end{proof}

In sequence, we adopt the following coordinate system on $A$ and $N^+$, \linebreak respectively,
\begin{equation}\label{coordinatesystem}
  \phi(g) = (a_1,a_2,\ldots,a_{n-1}), \ \ \ \psi(g) = (y_{12},y_{13}, \ldots,y_{n-1n}),
\end{equation}
which are global. We denote coordinate vector fields on $A$ and $N^+$, respectively, by $\frac{\partial}{\partial a_i}$ and $\frac{\partial}{\partial y_{kl}}$. Coordinate 1-forms on $A$ and $N^+$ are denoted by $da_i$ and $dy_{kl}$, respectively. We thus have
\[
  \delta A_t = da_i(\delta A_t)\dfrac{\partial}{\partial a_i} \ \ \textrm{and} \ \ \  \delta N_t = dy_{kl}(\delta N_t) \dfrac{\partial}{\partial y_{kl}}.
\]
Consequently,
\[
  L(A_t) = \int_{0}^t \omega \delta A_s = \int_0^t da_i(\delta A_s) E_i= A^{i}_s E_i,
\]
where $E_i$ is the matrix with values $1$ at $(j,j)$-entry, $-1$ at $(n,n)$-entry, for $j=1, \ldots n$, and $0$ and other entries. Here $\{E_1, \ldots, E_{n-1}\}$ is a basis of $\mathfrak{a}$.

We denote by $E_{kl}$ the matrix with values $1$ at $(k,l)$-entry and $0$ and other entries. It is clear that $\{E_{kl}; 1\leq k <l \leq n\}$ is a basis for $\mathfrak{n}^+$. Writing $a = (a^{ii}) \in A$ it follows that
\[
  Ad(a^{-1})(E_{kl}) =\dfrac{a^{ll}}{a^{kk}}E_{kl}.
\]
 Hence
\[
  L(N_t) = \int_0^t Ad(A_s^{-1})\omega \delta N_s = \left(\int_0^t \dfrac{A^{ll}_s}{A^{kk}_s}\delta N^{kl}_s\right) E_{kl}.
\]
It follows that
\begin{equation}\label{eqslnson2}
  L(Y_t) =  \sum_{i=1}^{n-1}A^{ii}_t E_i \oplus \sum_{1\leq k<l\leq n}\left(\int_0^t \dfrac{A^{ll}_s}{A^{kk}_s}\delta N^{kl}_s \right) E_{kl}.
\end{equation}

\begin{corollary}
  Let $X_t$ be a semimartingale in $Sl(n,{\Bbb R})/SO(n,{\Bbb R})$. Then $X_t$ is a martingale with respect to Canonical connection if and only if
  \[
    (A^{11}_t,\ldots, A^{(n-1)(n-1)}_t,-\sum_{i=1}^{n-1}A^{ii}_t) \in \mathbb{R}^n
  \]
  and
  \[
    \left(\int_0^t \dfrac{A^{22}_s}{A^{11}_s}\delta N^{12}_s,\ldots, \int_0^t \dfrac{A^{nn}(s)}{A^{11}_s}\delta N^{1n}_s, \ldots, \int_0^t \dfrac{A^{nn}_s}{A^{(n-1)(n-1)}_s}\delta N^{(n-1)n}_s\right) \in \mathbb{R}^{\frac{n(n-1)}{2}}
  \]
  are local martingales.
\end{corollary}

\end{document}